\documentclass[a4paper]{amsart}

\usepackage{etex}
\usepackage{aliascnt}
\usepackage{bbm} 
\usepackage{pbox}
\usepackage{booktabs}
\usepackage{arydshln}
\usepackage{pinlabel, caption, subcaption}
\usepackage{esvect}
\usepackage{cals}
\usepackage{verbatim}
\usepackage{multicol}
\newcommand{\Z}{\mathbb{Z}}

\newcommand{\Q}{\mathbb{Q}}
\newcommand{\C}{\mathbb{C}}

\usepackage{lmodern}
\usepackage{booktabs}
\usepackage[dvipsnames,svgnames,x11names,hyperref]{xcolor}
\usepackage{mathtools,url,graphicx,verbatim,amssymb,enumerate,stmaryrd,nicefrac}
\usepackage[pagebackref,colorlinks,citecolor=blue,linkcolor=blue,urlcolor=blue,filecolor=blue]{hyperref}
\usepackage{microtype}
\usepackage[margin=1.25in]{geometry}
\usepackage{pdflscape}

\usepackage{tikz, tikz-cd}
\usetikzlibrary{matrix,arrows,decorations.pathmorphing}
\usetikzlibrary{arrows,decorations.pathmorphing,backgrounds,fit,positioning,shapes.symbols,chains,calc}

\numberwithin{theoremcounter}{section}
\newaliascnt{theoremauto}{theoremcounter}

\newaliascnt{Defauto}{theoremcounter}

\newaliascnt{exampleauto}{theoremcounter}

\newaliascnt{lemmaauto}{theoremcounter}

\newaliascnt{propositionauto}{theoremcounter}

\newaliascnt{corollaryauto}{theoremcounter}

\newaliascnt{remarkauto}{theoremcounter}

\newaliascnt{notationauto}{theoremcounter}

\newaliascnt{claimauto}{theoremcounter}

\newaliascnt{warningauto}{theoremcounter}

\newaliascnt{questionauto}{theoremcounter}

\newaliascnt{discussionauto}{theoremcounter}

\newaliascnt{computationauto}{theoremcounter}

\newaliascnt{conjectureauto}{theoremcounter}

\newaliascnt{convauto}{theoremcounter}

\newtheorem{theorem}[theoremauto]{Theorem}
\newtheorem{lemma}[lemmaauto]{Lemma}
\newtheorem{proposition}[propositionauto]{Proposition}
\newtheorem{corollary}[corollaryauto]{Corollary}
\newtheorem*{corollary*}{Corollary}

\newtheorem{atheorem}{Theorem}

\theoremstyle{definition}
\newtheorem{definition}[Defauto]{Definition}

\theoremstyle{remark}

\newtheorem{remark}[remarkauto]{Remark}

\newcommand{\T}{\ensuremath\mathcal{T}}

\DeclareMathOperator{\GL}{GL}
\DeclareMathOperator{\SL}{SL}
\DeclareMathOperator{\St}{St}
\DeclareMathOperator{\Sh}{Sh}
\DeclareMathOperator{\id}{id}

\DeclareMathOperator{\Sym}{Sym}

\title{Hopf algebras, Steinberg modules, and the unstable cohomology of $\SL_n(\Z)$ and $\GL_n(\Z)$}
\renewcommand{\shorttitle}{Hopf algebras, Steinberg modules, and the unstable cohomology of $\SL_n(\Z)$}

\author{Avner Ash}
\email{avner.ash@bc.edu}
\address{Boston College\\
Chestnut Hill, MA 02445\\
USA}

 \author{Jeremy Miller}\thanks{Jeremy Miller was supported by NSF Grant DMS-2202943}
 \email{jeremykmiller@purdue.edu}  
\address{Purdue University \\
 	 West Lafayette IN, 47907 \\USA}

\author{Peter Patzt}
\email{ppatzt@ou.edu}
\address{University of Oklahoma\\
 	 Norman OK, 73019 \\USA}
	 \thanks{Peter Patzt was supported by a Simons Foundation Collaboration Grant}

\AtBeginDocument{%
	\def\MR#1{}}

\date{\today}

\begin{document}
	
\begin{abstract} 
We prove that the direct sum of all homology groups of the integral general linear groups with Steinberg module coefficients form a commutative Hopf algebra, in particular a free graded commutative algebra. We use this to construct new infinite families of unstable cohomology classes of $\SL_n(\Z)$.
\end{abstract}

\maketitle



\section{Introduction}

The goal of this paper is to prove a new structural result on the rational cohomology of general and special linear groups over the integers. Let $\St_n \Q$ denote the $n$th Steinberg module of the rationals (top degree reduced homology of the poset of subspaces of $\Q^n$) and $\det$ denote the determinant representation (unless stated otherwise, all homology will be taken with coefficients in a field of characteristic zero). By work of Borel--Serre \cite{BoSe} (see also Putman--Studenmund \cite{PutStu}), the twisted homology groups \[H_*(\GL_n \Z ;\St_n\Q) \text{ and } H_*(\GL_n \Z ;\St_n \Q \otimes \det) \] determine the cohomology  of $\SL_n(\Z)$ and $\GL_n(\Z)$:
\begin{itemize}
\item 
$H^{\binom n2 - k}(\SL_n \Z) \cong H_k(\GL_n \Z; \St_n \Q) \oplus  H_k(\GL_n \Z; \St_n \Q  \otimes \det)$ for $n>0$,
\item $H^{\binom n2 - k}(\GL_n \Z) \cong H_k(\GL_n \Z; \St_n \Q)$ for $n$ odd,
\item $H^{\binom n2 - k}(\GL_n \Z) \cong H_k(\GL_n \Z; \St_n \Q \otimes \det)$ for $n$ even.

\end{itemize}

There is a well-known equivariant product structure on the Steinberg modules (see e.g. \cite{MNP} and \cite{MPW}). In particular, there is a $\GL_n \Z \times \GL_m \Z$-equivariant map
\begin{equation}\label{Steinbergmult}
\St_n \Q \otimes \St_m \Q \longrightarrow \St_{n+m}\Q
\end{equation}
which induces ring structures on \[ H := \bigoplus_{k,n \ge 0} H_k(\GL_n \Z; \St_n \Q)\quad\text{and}\quad H^{\det} :=\bigoplus_{k,n \ge 0} H_k(\GL_n \Z;  \St_n \Q\otimes \det).\] We define coproduct structures on these rings which make them into Hopf algebras.

\begin{atheorem}\label{H(GL)}
$H$ and $H^{\det}$   have the structure of graded commutative Hopf algebras.

\end{atheorem}

Using a theorem of Leray, we deduce the following.

\begin{corollary}\label{introcor}
$H$ and $H^{\det}$  are free graded commutative algebras. 
\end{corollary}

Our freeness result lets us easily produce new nonzero classes in the cohomology of $\SL_n \Z$ and $\GL_n \Z$ simply by multiplying classes that were previously known to be nonzero. Moreover, using elementary results on Hopf algebras, one can check algebraic independence of cohomology classes using the coproduct. We apply our theory to the so-called wheel classes. Brown--Schnetz \cite[Theorem 4(i)]{Brown24} showed that there are nonzero classes 
\[w_n \in H_n(\GL_n \Z; \St_n \Q) \cong H^{\binom n2 - n}(\GL_n \Z) \subset H^{\binom n2 - n}(\SL_n \Z) \]
for every odd $n\ge 3$. For odd $n\ge1$, let
\[t_n \in H_{\binom n2}(\GL_n \Z; \St_n \Q)  \cong H^0(\GL_n \Z)\]
be a generator of the one-dimensional vector space. (It turns out that $t_3$ and $w_3$ both generate the one-dimensional vector space $H_3(\GL_3\Z; \St_3\Q)$.) And let, for even $n \ge 2$,
\[t_n \in H_{\binom n2}(\GL_n \Z;\St_n \Q\otimes \det)  \cong H^0(\GL_n \Z)\]
be a generator of the one-dimensional vector space.

\begin{atheorem}\label{wheelclasses}
The free graded commutative algebra 
\[ \bigwedge[t_1,t_5,t_9,\dots] \otimes \Sym[t_7,t_{11},\dots] \otimes \Sym[w_3,w_5,\dots]\]
 embeds as an algebra into $H$ and the free graded commutative algebra
\[ \bigwedge[t_2,t_6,t_{10} ,\dots] \otimes \Sym[t_4,t_8,\dots]\] 
embeds as an algebra into $H^{\det}$.
\end{atheorem}

The cohomology $H^*(\SL_n(\Z))$ has been completely computed for $n \leq 7$ (see \cite{soule,Horozov,EVGS}). Using \autoref{wheelclasses}, we describe bases of the vector spaces $H^k(\SL_n(\Z))$ for $n \leq 7$ in terms of the wheel classes, trivial classes, and the Borel classes $\sigma_{5} \in H^5(\SL_6 \Z)$, $\sigma_5\in H^5(\SL_7 \Z)$, and $\sigma_9 \in H^9(\SL_7 \Z)$, which are known to be nonzero (see \cite{Bor, Franke, BenaSanderJeff, Brown23}):

\begin{figure}[h!]

\begin{multicols}{2}

\begin{tikzpicture}
		\begin{scope}[scale=.8]
			\def\WW{.8}
			\def\HH{.7}
			\clip (-.75,-.75) rectangle ({\WW*7+.97},22.5*\HH);
			\draw (-.5,0)--({\WW*7+0.3},0) node[right] {$n$};
			\draw (0,-1) -- (0,21.3*\HH) node[above] {$k$};
			
			\begin{scope}
				\foreach \s in {0,...,21}
				{
					\draw [dotted] (-.5,\s*\HH)--({\WW*7+0.25},\s*\HH);
					\node [fill=white] at (-.25,\s*\HH) [left] {\tiny $\s$};
				}

				\foreach \s in {0,...,7}
				{
					\draw [dotted] ({\WW*\s},-0.5*\HH)--({\WW*\s},21.5*\HH);
					\node [fill=white] at ({\WW*\s},-.5) {\tiny $\s$};					
				}
			\end{scope}

\node [fill=white] at ({0*\WW},0) {$\Q$};
\node [fill=white] at ({1*\WW},0) {$\Q$};
\node [fill=white] at ({2*\WW},1*\HH) {$\Q$};
\node [fill=white] at ({3*\WW},3*\HH) {$\Q$};
\node [fill=white] at ({4*\WW},3*\HH) {$\Q$};
\node [fill=white] at ({4*\WW},6*\HH) {$\Q$};
\node [fill=white] at ({5*\WW},5*\HH) {$\Q$};
\node [fill=white] at ({5*\WW},10*\HH) {$\Q$};
\node [fill=white] at ({6*\WW},5*\HH) {$\Q$};
\node [fill=white] at ({6*\WW},6*\HH) {$\Q$};
\node [fill=white] at ({6*\WW},7*\HH) {$\Q$};
\node [fill=white] at ({6*\WW},10*\HH) {$\Q^2$};
\node [fill=white] at ({6*\WW},15*\HH) {$\Q$};
\node [fill=white] at ({7*\WW},6*\HH) {$\Q$};
\node [fill=white] at ({7*\WW},7*\HH) {$\Q$};
\node [fill=white] at ({7*\WW},12*\HH) {$\Q$};
\node [fill=white] at ({7*\WW},16*\HH) {$\Q$};
\node [fill=white] at ({7*\WW},21*\HH) {$\Q$};


		\end{scope}
	\end{tikzpicture}

\begin{tikzpicture}
		\begin{scope}[scale=.8]
			\def\WW{1}
			\def\HH{.7}
			\clip (-.75,-.75) rectangle ({\WW*7+.97},22.5*\HH);
			\draw (-.5,0)--({\WW*7+0.3},0) node[right] {$n$};
			\draw (0,-1) -- (0,21.3*\HH) node[above] {$k$};
			
			\begin{scope}
				\foreach \s in {0,...,21}
				{
					\draw [dotted] (-.5,\s*\HH)--({\WW*7+0.25},\s*\HH);
					\node [fill=white] at (-.25,\s*\HH) [left] {\tiny $\s$};
				}

				\foreach \s in {0,...,7}
				{
					\draw [dotted] ({\WW*\s},-0.5*\HH)--({\WW*\s},21.5*\HH);
					\node [fill=white] at ({\WW*\s},-.5) {\tiny $\s$};					
				}
			\end{scope}

\node [fill=white] at ({0*\WW},0) {$1$};
\node [fill=white] at ({1*\WW},0) {$t_1$};
\node [fill=white] at ({2*\WW},1*\HH) {$t_2$};
\node [fill=white] at ({3*\WW},3*\HH) {$w_3$};
\node [fill=white] at ({4*\WW},3*\HH) {$w_3t_1$};
\node [fill=white] at ({4*\WW},6*\HH) {$t_4$};
\node [fill=white] at ({5*\WW},5*\HH) {$w_5$};
\node [fill=white] at ({5*\WW},10*\HH) {$t_{5}$};
\node [fill=white] at ({6*\WW},5*\HH) {$w_{5}t_{1}$};
\node [fill=white] at ({6*\WW},6*\HH) {$w_{3}^2$};
\node [fill=white] at ({6*\WW},7*\HH) {$t_4t_2$};
\node [fill=white] at ({6*\WW},10*\HH) {$t_{5}t_{1}, \sigma_5$};
\node [fill=white] at ({6*\WW},15*\HH) {$t_6$};
\node [fill=white] at ({7*\WW},6*\HH) {$w_{3}^2t_{1}$};
\node [fill=white] at ({7*\WW},7*\HH) {$w_{7}$};
\node [fill=white] at ({7*\WW},12*\HH) {$\sigma_9$};
\node [fill=white] at ({7*\WW},16*\HH) {$\sigma_5$};
\node [fill=white] at ({7*\WW},21*\HH) {$t_{7}$};


		\end{scope}
	\end{tikzpicture}

\end{multicols}

\caption{$H_k(\SL_n \Z; \St_n \Q)$}
\label{table}

\end{figure}

\begin{remark}
Ash \cite{Ash24} recently proved that $t_1$ and $w_3$ generates a free graded commutative algebra inside of $H$. His paper inspired this collaboration and our coproducts on $H$ and $H^{\det}$ come from his product on cosharblies.
\end{remark}

\begin{remark}

After a draft of this paper had been circulated, we learned of a forthcoming paper by F. Brown, M. Chan, S. Galatius, and S. Payne where they independently constructed a commutative Hopf algebra structure on $H$ and use it to deduce many properties of the cohomology of $\SL_n(\Z$).

\end{remark}

\begin{remark}
The bialgebra structure on $H$ is the shadow of a bialgebra structure on the Steinberg modules themselves. Steinberg modules assemble to form a representation of the groupoid $\bigsqcup_n \GL_n(\Z)$. This category of representations carries two monoidal structures, one often called the Day convolution product considered for example in \cite{MNP} and one based on parabolic induction considered for example by Nagpal \cite{VIpart1}. These monoidal structures are compatible in the sense that they form a duoidal category. The Steinberg modules form an algebra with respect to the first product, a coalgebra with respect to the second product, and bialgebra in a duoidal sense. For the definitions of a duoidal category and a bialgebra therein, see B\"ohm--Chen--Zhang \cite{duoidalcat}. To keep this paper less abstract, we will only work with classical Hopf algebras. However, we plan to revisit these ideas to study the unstable cohomology of congruence subgroups.

\end{remark}

\begin{remark}
Free rings are determined by their indecomposables so it is natural to ask for a description of the indecomposables of $H$. It follows from \cite[Remark 1.6]{MPW} that these indecomposables are isomorphic to the reduced equivariant homology of Rognes' common basis complex \cite{Rog}.

\end{remark}

\subsection*{Acknowledgements}

We would like to thank Francis Brown, Benjamin Br\"uck, Melody Chan, S\o{}ren Galatius, Rohit Nagpal, Sam Payne, Andrew Snowden, and Robin Sroka for helpful conversations.

\section{A differential graded bialgebra that computes Steinberg homology}

Let $R$ be a PID and let $F$ be its field of fractions. Further, let $K$ be our coefficient ring, a field of characteristic zero.

The \emph{Tits building} of $F^n$ denoted by $\T_n(F)$ is the simplicial complex whose vertices are nonzero proper subspaces $0 \subsetneq V \subsetneq F^n$ and $\{V_0, \dots, V_p\}$ is a simplex if
\[ V_0 \subset \dots \subset V_p.\]

The Solomon--Tits Theorem \cite{Solomon} implies that $\T_n(F)$ is a wedge of $(n-2)$-dimensional spheres and its only nonzero reduced homology group is called the Steinberg module
\[ \St_n(F) := \tilde H_{n-2}(\T_n(F); K).\]
By convention, we define $\St_0(F)$ to be $K$.

Lee--Szczarba \cite{LS} gave a flat $K\GL_n(R)$-resolution of $\St_n(F)$. We use a slight variant of it here, which can be established with the same proof. See also Ash--Gunnells--McConnell \cite{AGM12}. Let $V_n$ be the set of all nonzero vectors in $R^n\subset F^n$ and $V^\pm_n$ be the quotient of $V_n$ identifying negatives. By abuse of notation, we denote an element of $V^\pm_n$ simply by its representative $v\in R^n$. Let $X_n$ be the simplicial complex whose vertices are the elements of $V^\pm_n$ and all (finite) sets of vertices are simplices. Let $L_n$ the subcomplex of all simplices $\{v_0,\dots, v_p\}$ such that the $F$-span of $v_0, \dots, v_p$ is not all of $F^n$. Clearly, $X_n$ is contractible and using Lee--Szczarba's argument it follows that $L_n$ is homotopy equivalent to $\T_n(F)$, which implies that
\[ H_i(X_n,L_n;K) \cong \tilde H_{i-1}(L_n; K) \cong \begin{cases} \St_n(F) & i = n-1\\ 0 &i \neq n-1\end{cases}.\]
Because $C_i(X_n,L_n;K) = 0$ for $i< n-1$, we get an exact sequence
\[ \dots \longrightarrow C_n(X_n,L_n;K) \longrightarrow C_{n-1}(X_n,L_n;K) \longrightarrow \St_n(F) \longrightarrow 0.\]
\cite[Lemma 3.2]{CP} implies that $C_i(X_n,L_n;K)$ is a flat $K\GL_n(R)$-module.

We will use the notation
\[ \Sh_{n,k} = C_{n+k-1}(X_n,L_n;K)\]
and refer to $\Sh_{n,*} \to \St_n(F)$ as the \emph{Sharbly resolution} of $\St_n(F)$. Given nonzero vectors $v_1, \dots, v_{n+k}\in R^n$, we get a canonical element
\[ [v_1, \dots, v_{n+k}] \in \Sh_{n,k}\]
given by the corresponding simplex $\{v_1, \dots, v_{n+k}\}$ in $X_n$. These elements are nonzero if and only if the $F$-span of $v_1, \dots, v_{n+k}$ is  $F^n$. We call such an element then a \emph{basic sharbly}.

Let $\chi\colon R^\times \to K^\times$ be a character (i.e. a group homomorphism), which we will view as a one-dimensional $K\GL_n(R)$-module that factors through the determinant. We will use the notation
\[ [v_1, \dots, v_{n+k}]^\chi := [v_1, \dots, v_{n+k}] \otimes 1 \in \Sh^\chi_{n,k}:= \Sh_{n,k} \otimes_K \chi.\]

We will further consider the $\GL_n(R)$-coinvariants of $\Sh^\chi_{n,k}$ and call its elements \emph{coinvariant sharblies}. In particular, we will denote the image of a basic sharbly $ [v_1, \dots, v_{n+k}]^\chi\in \Sh^\chi_{n,k}$ by
\[ [v_1, \dots, v_{n+k}]^\chi_{\GL_n(R)} \in (\Sh^\chi_{n,k})_{\GL_n(R)}\]
and call it a \emph{basic coinvariant sharbly}. Because $\GL_n(R)$ acts on the set of all basic sharblies, which up to signs%
\footnote{``Up to signs'' means that while $[v_1, \dots, v_{n+k}]^\chi$ is antisymmetric, if we choose a representative from each $S_{n+k}$-orbit of $(n+k)$-tuples of vectors, we get a basis of $\Sh^\chi_{n,k}$.
}
 form a basis of $\Sh^\chi_{n,k}$, the set of nonzero basic coinvariant sharblies form a basis up to signs. 

Let us denote by
\[ B^\chi_{n,k} :=  (\Sh^\chi_{n,k})_{\GL_n(R)}\quad\text{and}\quad B^\chi_m := \bigoplus_{n+k = m} B^\chi_{n,k}\]
the coinvariants of the sharblies and by
\[ B^\chi :=  \bigoplus_{m\ge0} B^\chi_m = \bigoplus_{n,k\ge0} B^\chi_{n,k}\]
the (bi-)graded vector space of all coinvariant sharblies. It comes with a differential 
\begin{align*}
 \partial\colon B^\chi_{n,k} = (\Sh^\chi_{n,k})_{\GL_n(R)} &\longrightarrow (\Sh^\chi_{n,k-1})_{\GL_n(R)} = B^\chi_{n,k-1}\\
 [v_1,\dots, v_{n+k}]^\chi_{\GL_n(R)} &\longmapsto \sum_{i=1}^{n+k}  (-1)^i[v_1,\dots,\hat v_i,\dots, v_{n+k}]^\chi_{\GL_n(R)}
 \end{align*}
of degree $-1$, such that its homology is
\[ H(B^\chi)_{n,k} \cong H_k(\GL_n(R);\St_n(F) \otimes_K \chi).\]

\begin{definition}
We define 
\[ \nabla \colon B^\chi_{n,k} \otimes B^\chi_{m,l} \longrightarrow B^\chi_{n+m,k+l}\]
by sending a tensor product of basic coinvariant sharblies $x = [v_1, \dots, v_{n+k}]^\chi_{\GL_n(R)}$ and $y =  [w_1, \dots, w_{m+l}]^\chi_{\GL_m(R)}$ to
\[ \nabla(x \otimes y) = xy = [ \phi_{n,n+m}(v_1), \dots, \phi_{n,n+m}(v_{n+k}), \psi_{m,n+m}(w_1), \dots, \psi_{m,n+m}(w_{m+l})]^\chi_{\GL_{n+m}(R)},\]
where $\phi_{n,n+m}\colon R^n \to R^{n+m}$ sends the standard basis vector $e_i \in R^n$ to $e_i \in R^{n+m}$ and $\psi_{m,n+m} \colon R^m \to R^{n+m}$ sends $e_i\in R^m$ to $e_{n+i} \in R^{n+m}$.
\end{definition}

The following lemma is easy to check.

\begin{lemma}
$\nabla \colon B^\chi \otimes B^\chi \to B^\chi$ defines a (bi-)graded algebra.
\end{lemma}

We call $\chi\colon R^\times \to K^\times$ \emph{odd} if $\chi(-1) = -1$ and \emph{even} if $\chi(-1) =1$.

\begin{lemma}
If $\chi$ is odd and $n$ is odd, then $B^\chi_{n,k} =0$.
\end{lemma}

\begin{proof}
Let $x=[v_1, \dots, v_{n+k}]^\chi$ be a basic sharbly in $\Sh^\chi_{n,k}$. Then
\[ (-\id)\cdot x = \chi(-1)^n \cdot [-v_1, \dots, -v_{n+k}]^\chi = - [v_1, \dots, v_{n+k}]^\chi = -x\]
and thus vanishes in the coinvariants $B^\chi_{n,k}$.
\end{proof}

\begin{lemma}\label{commutative}
$B^\chi$ is graded commutative, i.e. for $x\in B^\chi_{n,k}$ and $y\in B^\chi_{m,l}$, we get
\[ xy = (-1)^{(n+k)(m+l)} yx.\]
\end{lemma}

\begin{proof}
Let $x = [v_1, \dots, v_{n+k}]^\chi_{\GL_n(R)}$ and $y =  [w_1, \dots, w_{m+l}]^\chi_{\GL_m(R)}$ be basic coinvariant sharblies. We observe that
\begin{multline*}xy = [ \phi_{n,n+m}(v_1), \dots, \phi_{n,n+m}(v_{n+k}), \psi_{m,n+m}(w_1), \dots, \psi_{m,n+m}(w_{m+l})]^\chi_{\GL_{n+m}(R)} \\= (-1)^{(n+k)(m+l)} [  \psi_{m,n+m}(w_1), \dots, \psi_{m,n+m}(w_{m+l}),\phi_{n,n+m}(v_1), \dots, \phi_{n,n+m}(v_{n+k})]^\chi_{\GL_{n+m}(R)}. \end{multline*}
Let $T_{n,m}$ be the permutation matrix sending $e_i$ to $e_{m+i}$ if $i\le n$ and to $e_{i-n}$ if $i>n$. We can use $T_{n,m}\in \GL_{n+m}(R)$ to swap $\phi$ and $\psi$ in the sense that 
\[ T_{n,m}\phi_{n,n+m} = \psi_{n,n+m}\quad\text{and}\quad T_{n,m}\psi_{m,n+m} = \phi_{m,n+m}.\]
The determinant of $T_{n,m}$ is $(-1)^{nm}$, so in particular $1$ if $n$ or $m$ are even. We then get that
\begin{multline*} [  \psi_{m,n+m}(w_1), \dots, \psi_{m,n+m}(w_{m+l}),\phi_{n,n+m}(v_1), \dots, \phi_{n,n+m}(v_{n+k})]^\chi_{\GL_{n+m}(R)} \\= T_{n,m} \chi(-1)^{nm}   [  \phi_{m,n+m}(w_1), \dots, \phi_{m,n+m}(w_{m+l}),\psi_{n,n+m}(v_1), \dots, \psi_{n,n+m}(v_{n+k})]^\chi_{\GL_{n+m}(R)} \\=T_{n,m} \chi(-1)^{nm} yx. \end{multline*}
Therefore 
\[ xy = (-1)^{(n+k)(m+l)} \cdot \chi(-1)^{nm} \cdot yx\]
in the coinvariants $B^\chi_{n+m,k+l}$. The factor $\chi(-1)^{nm}$ is of course $1$ if $\chi(-1)=1$, and it is also $1$ if $\chi(-1)=-1$ because then $nm$ is even or $xy=yx=0$.
\end{proof}

\begin{lemma}
$B^\chi$ together with the differential $\partial$ is a differential graded algebra.
\end{lemma}

\begin{proof}
Given basic coinvariant sharblies $x = [v_1, \dots, v_{n+k}]^\chi_{\GL_n(R)}$ and $y =  [w_1, \dots, w_{m+l}]^\chi_{\GL_m(R)}$, we need to prove that
\[ \partial(xy) = (\partial x)y + (-1)^{n+k} x(\partial y).\]
This follows as
\begin{align*} &\partial(xy) \\= &\phantom{+}\sum_{i=1}^{n+k}  
(-1)^i
[ \phi_{n,n+m}(v_1), \dots,\widehat{ \phi_{n,n+m}(v_i)},\dots, \phi_{n,n+m}(v_{n+k}), \psi_{m,n+m}(w_1), \dots, \psi_{m,n+m}(w_{m+l})]^\chi_{\GL_{n+m}(R)} \\
&+ (-1)^{n+k} \sum_{i=1}^{m+l}  
(-1)^i
[ \phi_{n,n+m}(v_1),\dots, \phi_{n,n+m}(v_{n+k}), \psi_{m,n+m}(w_1), \dots,\widehat{ \psi_{m,n+m}(w_i)},\dots, \psi_{m,n+m}(w_{m+l})]^\chi_{\GL_{n+m}(R)}\\
= & \phantom{+}(\partial x)y + (-1)^{n+k} x(\partial y)\qedhere.
\end{align*}
\end{proof}

We now want to describe the coalgebra structure on $B^\chi$. We will actually first describe it before taking coinvariants. 

\begin{definition}
Let $U$ be an $a$-dimensional subspace of $F^n$. Let $X(U)$ be the simplicial complex whose vertices are 
\[ V^\pm(U) := \{ v \in V^\pm_n \mid v \in U\}\]
and every finite subset forms a face. Let $L(U)$ the subcomplex of $X(U)$ of all simplices $\{v_0, \dots, v_p\}$ such that the $F$-span of $v_0, \dots, v_p$ is not all of $U$.
We define 
\[\Sh^\chi_s(U):= C_{a+s-1}(X(U), L(U); K).\]
Note that a basis of the free $R$-module $U\cap R^n$ gives an isomorphism $\Sh^\chi_*(U) \cong \Sh^\chi_{a,*}$ of chain complexes.

Further, let $W$ be a $b$-dimensional subspace of $U$.
We define $\Sh^\chi_s(U/W)$ as the quotient of $\Sh^\chi_{s-b}(U)$ by identifying basic sharblies $[v_1, \dots, v_{a+(s-b)}]^\chi$ and $[w_1, \dots, w_{a+(s-b)}]^\chi$ in  $\Sh^\chi_{s-b}(U)$ if $v_i -w_i \in W$ for all $i$. We also identify a basic sharbly $[v_1, \dots, v_{a+(s-b)}]^\chi$ with zero if $v_i \in W$ for some $i$ or if $v_i-v_j \in W$ for some $i\neq j$. Note that a basis of the free $R$-module $(U\cap R^n)/(W\cap R^n)$ gives an isomorphism $\Sh^\chi_*(U/W) \cong \Sh^\chi_{a-b,*}$ of chain complexes.
\end{definition}

\begin{definition}
Let $W \subseteq U \subseteq F^n$ be subspaces of $F^n$. We define a map 
\[ \Delta\colon \Sh^\chi_{k}(U/W) \longrightarrow \bigoplus_{\substack{W\subseteq V \subseteq U\\p+q=k}} \Sh^\chi_{p}(V/W) \otimes \Sh^\chi_{q}(U/V)\]
in the following way. 

Let $x = [v_1, \dots, v_{\dim U - \dim W+k}]^\chi\in \Sh^\chi_{k}(U/W)$ be a basic sharbly. We define $\mathcal U_x$ to be the set of subspaces $V\subseteq U$ that are spanned by $W$ and a subset of $\{v_1, \dots, v_{\dim U- \dim W+k}\}$. ($W$ and $U$ are necessarily included in $\mathcal U_x$.) Let $[m]$ denote the set $\{1,\ldots,m\}$. For $V \in \mathcal U_x$, let 
\[S(V) = \{ i \in [\dim U - \dim W+k] \mid v_i \in V\}.\]
 Let $s_1<\dots <s_{\dim V - \dim W+p}$ be the elements of $S(V)$ and $t_1< \dots< t_{\dim U-\dim V +q}$ the elements of $[a+k] \setminus S(V)$. 

We set
\[ x'_{V} =  [ v_{s_1},\dots, v_{s_{\dim V - \dim W+p}}]^\chi \in \Sh^\chi_p(V/W)\]
and
\[x''_{V} =  [ v_{t_1},\dots, v_{t_{\dim U-\dim V +q}}]^\chi \in \Sh^\chi_q(U/V).\]

We define
\[ \Delta(x) = \sum_{V \in \mathcal U_x} (-1)^{\pi_{V/W\subset U/W}} x'_{V} \otimes x''_{V},\]
where $\pi_{V/W\subset U/W}$ is the shuffle permutation sending $i$ to $s_i$ if $i \le \dim V - \dim W+ p$ and to $t_{i-\dim W-p}$ if $i>\dim V - \dim W +p$. If $U=F^n$ and $W=0$, we will sometimes simply write $\pi_V$ for $\pi_{V/W \subset F^n/W}$.
\end{definition}

\begin{lemma}
The map
\[ \Delta\colon \Sh^\chi_{n,k} \longrightarrow \bigoplus_{\substack{U \subset F^n\\p+q=k}} \Sh^\chi_{p}(U) \otimes \Sh^\chi_{q}(F^n/U)\]
is $\GL_n(R)$-equivariant and the $\GL_n(R)$-coinvariants of it give a map
\[ \Delta\colon B^\chi_{n,k} \longrightarrow \bigoplus_{\substack{a+b=n\\p+q=k}} B^\chi_{a,p} \otimes B^\chi_{b,q}.\]
\end{lemma}

\begin{proof}
For the equivariance it is important to note that $\GL_n(R)$ permutes the summands by acting on the set of subspaces $U\subset F^n$. 

The coinvariants on the left are $B^\chi_{n,k}$ by definition. On the right, we first note that $\GL_n(R)$ acts transitively on the set of $a$-dimensional $U\subset F^n$. (Here we use that $R$ is a PID.) Let $P_{a,b}$ denote the stabilizer of $F^a \subset F^n$ in $\GL_n(R)$, where $b=n-a$.

There is a canonical isomorphism of $\Sh^\chi_p(F^a) \to \Sh^\chi_{a,p}$. There is also a canonical isomorphism of $\Sh^\chi_q(F^n/F^a) \to \Sh^\chi_{b,q}$ by only recording the last $q$ coordinates of the representatives of vectors. $P_{a,b}$ acts on 
\[\Sh^\chi_{p}(F^a) \otimes \Sh^\chi_{q}(F^n/F^a) \cong \Sh^\chi_{a,p} \otimes \Sh^\chi_{b,q}\]
 via the projection onto its Levi factor: $P_{a,b} \to \GL_a(R) \times \GL_b(R)$. And that means that the coinvariants on the right-hand-side are isomorphic to
\[\bigoplus_{\substack{a+b=n\\p+q=k}} B^\chi_{a,p} \otimes B^\chi_{b,q}.\qedhere\]
\end{proof}

We get that $B^\chi$ is a coalgebra by the following lemma.

\begin{lemma}
$\Delta\colon B^\chi \to B^\chi \otimes B^\chi$ defines a (bi-)graded coalgebra.
\end{lemma}

\begin{proof}
We will only check coassociativity here and leave all other properties to the reader. In fact, we check coassociativity before taking coinvariants.

Let $x = [v_1, \dots, v_{n+k}]^\chi\in \Sh^\chi_{n,k}$ a basic sharbly. We compute that
\[ (\Delta \otimes \id) \circ \Delta (x) = \sum_{\substack{W,U \in \mathcal U_x\\W \subseteq U}}
(-1)^{\pi_{U \subset F^n}}
 ((-1)^{\pi_{W\subset U}} (x'_{U})'_W \otimes (x'_{U})''_W ) \otimes x''_{U},\]
where $(x'_{U})'_W \in \Sh^\chi_a(W)$, $(x'_{U})''_W \in  \Sh^\chi_ b(U/W)$, and $x''_{U} \in \Sh^\chi_c(F^n/U)$ for corresponding $a+b+c=k$.

We also compute that
\[ (\id \otimes \Delta) \circ \Delta (x) = \sum_{\substack{W,U \in \mathcal U_x\\W \subseteq U}}
(-1)^{\pi_{W\subset F^n}}
 x'_{W} \otimes ((-1)^{\pi_{U/W \subset F^n/W}} (x''_{W})'_{U/W}  \otimes (x''_{W})''_{U/W}),\]
where $x'_W \in \Sh^\chi_a(W)$, $(x''_{W})'_{U}  \in \Sh^\chi_ b(U/W)$, and $(x''_{W})''_{U} \in \Sh^\chi_c(F^n/U)$ for corresponding $a+b+c=k$.

It is now easy to see that
\[ (x'_{U})'_W = x'_W , \quad (x'_{U})''_W = (x''_{W})'_{U}, \quad x''_{U} = (x''_{W})''_{U},\text{and} \quad \pi_{W\subset U} \circ \pi_{U \subset F^n}= \pi_{U/W \subset F^n/W}\circ \pi_{W\subset F^n} .\qedhere\]
\end{proof}

Let us check that it behaves well with the differential.

\begin{lemma}
$B^\chi$ together with the differential $\partial$ is a differential graded coalgebra.
\end{lemma}

\begin{proof}
Given a basic sharbly $x = [v_1, \dots, v_{n+k}]^\chi \in \Sh^\chi_{n,k}$, we need to check that
\[ \Delta(\partial x) = \sum_{U \in \mathcal U_x} (-1)^{\pi_{U}} ( \partial x'_U \otimes x''_U  + (-1)^{|S(U)|} x'_U\otimes \partial x''_U).\]
 
 We want to show that both sides of this equation are equal to the sum 
 \[ \sum_{\substack{U \in \mathcal U_x\\i\in[n+k]}} (-1)^{\pi_{U}}\cdot (-1)^{\pi_U^{-1}(i)} d_i (x'_U \otimes x''_U),\]
where $d_i$ omits the vector $v_i$ from $x'_U$ if $v_i\in U$ and from $x''_U$ if $v_i \not\in U$.

On the right-hand-side, we can index the summands by $U \in \mathcal U_x$ and the $v_i$ that is being omitted. As for the sign, note that $\pi_U^{-1}(i)$ is the position of $v_i$ in $x'_U$ if $v_i \in U$ and the position of $v_i$ in $x''_U$ plus $|S(U)|$ if $v_i \not\in U$.

For the left-hand-side, let us compute $\Delta(x_i)$ for every
\[ x_i =  [v_1,\dots,\hat v_i,\dots, v_{n+k}]^\chi.\]
This is
\[ \Delta(x_i) = \sum_{U \in \mathcal U_{x_i}} (-1)^{\pi_{U,i}} (x'_i)_{U} \otimes (x''_i)_{U},\]
where $\pi_{U,i}$ is the shuffle permutation obtained from $\pi_U$ by removing the strand that comes from $i$. 
That means we get a summand $d_i (x'_U \otimes x''_U)$ for every pair $(U,i)$ such that $U \in \mathcal U_{x_i}$. For these summands, the sign of $\pi_{U,i}$ is equal to $(-1)^{|i-\pi_U^{-1}(i)|}$.

We finish the proof by observing that if $U \in \mathcal U_x \setminus \mathcal U_{x_i}$, then $d_i( x'_U \otimes x''_U) =0$. This follows because then $v_i \in U$ and the vectors $v_s$ with $s \in S(U) \setminus \{i\}$ do not span $U$ without $v_i$.
\end{proof}

We conclude this section by proving that the product and the coproduct are compatible.

\begin{lemma}
The diagram
\[
\begin{tikzcd}
B^\chi \otimes B^\chi \arrow[r, "\nabla"] \arrow[d, "\Delta \otimes \Delta"']       & B^\chi \arrow[r, "\Delta"] & B^\chi\otimes B^\chi                                                          \\
B^\chi \otimes B^\chi \otimes B^\chi \otimes B^\chi \arrow[rr, "\id\otimes \tau \otimes \id"] &                       & B^\chi \otimes B^\chi \otimes B^\chi \otimes B^\chi \arrow[u, "\nabla \otimes \nabla"']
\end{tikzcd}
\]
commutes. Here $\tau\colon B_n^\chi \otimes B_m^\chi \to B_m^\chi \otimes B_n^\chi$ is given by
\[ \tau(x \otimes y) = (-1)^{nm}( y \otimes x).\]
\end{lemma}

\begin{proof}
We evaluate both paths on $x\otimes y$ for basic sharblies $x = [v_1, \dots, v_{n+k}]^\chi \in \Sh^\chi_{n,k}$ and $y =  [w_1, \dots, w_{m+l}]^\chi\in \Sh^\chi_{m,l}$.

For the lower path we get that $x\otimes y$ is sent to
\[ \sum_{\substack{U \in \mathcal U_x\\ V \in \mathcal U_y}} (-1)^{pq}(-1)^{\pi_U}(-1)^{\pi_V}x'_Uy'_V \otimes x''_Uy''_V,\]
where $p$ is the number of vectors in $x''_U$ and $q$ the number of vectors in $y'_V$.
We note that we consider the product $x'_Uy'_V$ to be in
\[ \Sh^\chi_{a+c}(\phi(U) \oplus \psi(V)) \subset \Sh^\chi_{n+m,a+c},\]
where $a = |S(U)|-\dim U$ and $c = |S(V)|-\dim V$. And we consider the product $x''_Uy''_V$ to be in the quotient
\[ \Sh^\chi_{b+d}(F^{n+m}/(\phi(U) \oplus \psi(V)))\]
of $\Sh^\chi_{n+m,b+d}$, where $b= k-a$ and $d = l-c$.

For the upper path, we get that $x\otimes y$ is sent to
\[ \sum_{W \in \mathcal U_{xy}} (-1)^{\pi_{W}} (xy)'_W \otimes (xy)''_W.\]
To see that these two sums are the same, we first observe that every $W \in \mathcal U_{xy}$ uniquely decomposes into a direct sum $\phi(U)\oplus \psi(V) = W$ for pairs $(U,V) \in \mathcal U_x \times \mathcal U_y$.

It remains to show that the corresponding summands agree. For this note that 
\[ (xy)'_W = [\phi(v_{s_1}), \dots, \phi(v_{s_{\dim U + a}}), \psi(w_{\eta_1}), \dots, \psi(w_{\eta_{\dim V+c}})] \in \Sh^\chi_{a+c}(W),\]
where $S(U) = \{ s_1< \dots< s_{\dim U + a}\}$ and $S(V) = \{\eta_1<\dots < \eta_{\dim V + b}\}$ which is the same as
\[ x'_Uy'_V = [\phi(v_{s_1}), \dots, \phi(v_{s_{\dim U + a}}), \psi(w_{\eta_1}), \dots, \psi(w_{\eta_{\dim V+c}})] \in \Sh^\chi_{a+c}(\phi(U) \oplus \psi(V)).\]
And
\[ (xy)''_W = [\phi(v_{t_1}), \dots, \phi(v_{t_{n-\dim U+b}}), \psi(w_{\theta_1}), \dots, \psi(w_{\theta_{m-\dim V+b}})] \in \Sh^\chi_{b+d}(F^{n+m}/W),\]
where $[n+k]\setminus S(U) = \{ t_1< \dots< t_{n-\dim U+b}\}$ and $[m+l]\setminus S(V) = \{\theta_1<\dots < \theta_{m-\dim V+b}\}$, which is the same as
\[ x''_Uy''_V = [\phi(v_{t_1}), \dots, \phi(v_{t_{n-\dim U+b}}), \psi(w_{\theta_1}), \dots, \psi(w_{\theta_{m-\dim V+b}})] \in \Sh^\chi_{b+d}(\phi(F^n/U)\oplus \psi(F^m/V)).\]

Finally, we need to compare the signs. Let us consider the two permutations $\pi_U$ and $\pi_V$ as one permutation $\pi_{U,V}$ on $ n+k+m+l$ letters. We note that $\pi_W$ is $\pi_{U,V}$ postcomposing with the permutation that swaps the number of vectors in $x''_U$, which we denoted by $p$, and the number of vectors in $y'_V$, which we denoted by $q$, with each other. Thus
\[ (-1)^{\pi_{W}} = (-1)^{pq}(-1)^{\pi_U}(-1)^{\pi_V}.\qedhere\]
\end{proof}

The result of this section is summarized in the following theorem.

\begin{theorem}
For every PID $R$ and every character $\chi\colon R^\times \to \C^\times$, 
\[ B^\chi = \bigoplus_{n,k\ge0} (\Sh^\chi_{n,k} \otimes_K \chi)_{\GL_n(R)}\]
is a differential graded commutative bialgebra.
\end{theorem}

We get the following corollary as an immediate consequence.

\begin{corollary}\label{bialg}
For every PID $R$ and every character $\chi\colon R^\times \to \C^\times$, 
\[ H(B^\chi) \cong \bigoplus_{n,k\ge0} H_k(\GL_n(R); \St_n(F) \otimes_K \chi)\]
is a graded commutative bialgebra.
\end{corollary}

\begin{remark}
Let $R$ be a PID with a finite group of units $R^\times$, for example $R= \Z$ or $R$ a complex quadratic number ring which is a PID. And let $K$ be a field of characteristic zero over which the $R^\times$-representations split into one-dimensional irreducibles, for example $\C$. Then for $n\ge 1$, 
\[ H_k(\SL_n(R); \St_n(F)) \cong H_k(\GL_n(R); \St_n(F) \otimes \chi_1) \oplus \dots \oplus H_k(\GL_n(R); \St_n(F) \otimes \chi_r),\]
where the regular representation $KR^\times \cong \chi_1 \oplus \dots \oplus \chi_r$.

In fact, Equation \eqref{Steinbergmult} induces an algebra structure on
\[ \bigoplus_{n,k\ge 0} H_k(\SL_n(R); \St_n(F)),\]
but products between elements of summands corresponding to different characters are zero.
\end{remark}

\section{Hopf algebras}

In this section, we recall several definitions and facts about Hopf algebras and then prove a generalization of \autoref{H(GL)}.

A graded bialgebra $B = \bigoplus_{n\ge 0} B_n$ is called \emph{connected} if $B_0 = K$.

Usually, a Hopf algebra is defined to be a bialgebra that has an antipode. We will mainly cite Milnor--Moore \cite{MM} in this section, who do not require an antipode, i.e. their ``Hopf algebras'' are simply bialgebras. But it turns out that every connected bialgebra is automatically a Hopf algebra which follows from \cite[Proposition 8.2]{MM}. Together with \autoref{bialg}, this implies the following generalization of \autoref{H(GL)}.

\begin{theorem}\label{hopfalg}
For every PID $R$ and every character $\chi\colon R^\times \to \C^\times$, 
\[ H(B^\chi) \cong \bigoplus_{n,k\ge0} H_k(\GL_n(R); \St_n(F) \otimes_K \chi)\]
is a graded commutative Hopf algebra.
\end{theorem}

The \emph{primitives} of a bialgebra $B$ are given by
\[ P(B) = \{ x \in B \mid \Delta(x) = x \otimes 1 + 1 \otimes x \}.\]

For a connected bialgebra $B$, let $B_+$ denote the augmentation ideal $\bigoplus_{n\ge 1} B_n$. The \emph{indecomposables} of $B$ are
\[ Q(B) = B_+/(B_+)^2.\]

Because $P(B) \subset B_+$, there is a canonical map $P(B) \to Q(B)$.

\begin{proposition}[{Milnor--Moore \cite[Proposition 4.17]{MM}}]\label{inj}
A connected Hopf algebra $H$ over a characteristic zero field is commutative if and only if $P(H) \to Q(H)$ is injective.
\end{proposition}

\begin{theorem}[{Leray \cite[Theorem 7.5]{MM}}]\label{Leray}
A connected graded commutative Hopf algebra $H$ over a characteristic zero field is isomorphic to the free graded commutative algebra generated by $Q(H)$. 
\end{theorem}

Therefore we get the following generalization of \autoref{introcor}.

\begin{corollary}
For every PID $R$ and every character $\chi\colon R^\times \to \C^\times$, 
\[ H(B^\chi) \cong \bigoplus_{n,k\ge0} H_k(\GL_n(R); \St_n(F) \otimes_K \chi)\]
is a free graded commutative algebra.
\end{corollary}

\section{New classes}

In this section, we take $R=\Z$, $F=\Q$, and $K=\Q$. 
Let $n\ge 3$. We call the basic coinvariant sharbly
\[ w_n = [e_1, \dots, e_n, e_1-e_2, \dots, e_{n-1}-e_n, e_n-e_1]_{\GL_n(\Z)} \in (\Sh_{n,n})_{\GL_n(\Z)}\]
a \emph{wheel sharbly}. It is easy to check that $w_n =0$ if $n$ is even and that $\partial w_n = 0$ if $n$ is odd. Recently, Brown--Schnetz \cite[Theorem 4(i)]{Brown24} proved%
\footnote{They show that a region $\tau_{W_n}$ in the space $\mathcal P_n$ of positive definite $n\times n$ real matrices gives rise to a nonzero class in the locally finite homology $H^{\mathrm{lf}}_{2n}(\mathcal P_n/\GL_n(\Z)) \cong H_n(\GL_n(\Z);\St_n(\Q))$. There is a quasi-isomorphism from $(\Sh_{n,*})_{\GL_n(\Z)}$ to the locally finite chains of $\mathcal P_n/\GL_n(\Z)$ defined by sending a basic coinvariant sharbly $[v_1,\dots, v_{n+k}]_{\GL_n(\Z)}$ to the region in $\mathcal P_n$ which is the given by elements $c_1 \cdot v_1v_1^T + \dots + c_{n+k} \cdot v_{n+k}v_{n+k}^T$ for all positive real numbers $c_1, \dots, c_{n+k}>0$. Their discussion makes clear that our $w_n$ is sent to their $\tau_{W_n}$.}
that for odd $n\ge3$, $w_n$ gives rise to a nonzero element in homology
\[ H(B)_{n,n} \cong H_n(\GL_n(\Z); \St_n(\Q)).\]

We want to prove the following two proposition about them and the trivial classes $t_n \in H^0(\GL_n(\Z))$.

\begin{proposition}\label{primitivewheels}
$w_n$ is primitive in $B$.
\end{proposition}

\begin{proposition}\label{primitivetrivial}
$t_n$ is primitive in $H$ or $H^{\det}$, depending on the parity.
\end{proposition}

Because we know that $w_n$ is nonzero in $H(B)$, it follows that it is a nonzero primitive in $H(B)$. Using \autoref{inj} and \autoref{Leray}, \autoref{wheelclasses} follows as an immediate corollary.

\begin{proof}[Proof of \autoref{primitivewheels}]
For notational convenience, set $x=w_n$.
Recall that we defined the coproduct using the formula
\[ \Delta(x) = \sum_{U \in \mathcal U_{x}} (-1)^{\pi_{U}} x'_U \otimes x''_U.\]
We want to prove that if $U\in \mathcal U_x$ is a nonzero proper subspace of $F^n$, then $x''_U = 0$. Let $\beta$ denote reduction modulo $U$.

First assume that none of the $e_1, \dots, e_n$ are in $U$. Because $U$ is nonzero, without loss of generality  it contains  $e_1-e_2$. But that means that $\beta(e_1) = \beta(e_2)$ because $e_1-e_2$ is in the kernel of $\beta$ and thus $x''_U = 0$.

We can therefore assume without loss of generality that $e_1 \in U$. Let $k\ge 1$ be so that $e_1, \dots, e_k \in U$ and $e_{k+1} \not\in U$. This implies that $e_k-e_{k+1} \not\in U$, because $e_k \in U$. But $\beta(e_k- e_{k+1}) = -\beta(e_{k+1})$ because $e_k$ is in the kernel of $\beta$. Therefore $x''_U=0$.

This implies that $x''_U$ is only nonzero if $U=0$ or $U=F^n$ and we get that
\[ \Delta(x) = 1\otimes x + x\otimes 1.\qedhere\]
%
%
%
%
%
%
%
%
%
\end{proof}

\begin{proof}[Proof of \autoref{primitivetrivial}]
By Borel--Serre \cite{BoSe}, 
\[ H_k(\SL_n(\Z);\St_n(\Q)) \cong H_k(\GL_n(\Z); \St_n(\Q))\otimes H_k(\GL_n(\Z); \St_n(\Q) \otimes \det)\cong0\]
for $k>\binom n2$. And
\[ \binom a2 + \binom b2 \le \binom{a+b}2\]
with equality only if $a$ or $b$ is zero. This means that the coproduct $\Delta(t_n)$ of $t_n \in H_{\binom n2}(\GL_n(\Z);\St_n(\Q) \otimes \chi)$, where $\chi$ is trivial if $n$ is odd and the determinant representation if $n$ is even, can only be $1\otimes t_n + t_n \otimes 1$ for degree reasons.
\end{proof}

\bibliographystyle{amsalpha}
\bibliography{refs}

\vspace{.5cm}

\end{document}